\documentclass[a4paper,10pt]{amsart}

\usepackage{amsmath}
\usepackage{amsthm}
\usepackage{amssymb}
\usepackage{amsfonts}
\usepackage{mathrsfs}  
\usepackage{enumitem} 
\usepackage{xcolor}
\usepackage{soul}
\usepackage[bookmarks=false,hyperindex,pdftex,colorlinks,citecolor=blue,urlcolor=cyan]{hyperref}
\usepackage[a4paper,lmargin=3cm,rmargin=3cm,tmargin=4cm,bmargin=4cm,marginparwidth=2.8cm,marginparsep=1mm]{geometry} 

\DeclareMathOperator{\lspan}{span}                          
\DeclareMathOperator{\diam}{diam}                           
\DeclareMathOperator{\rad}{rad}                             
\DeclareMathOperator{\Lip}{Lip}                             

\newcommand{\NN}{\mathbb{N}}                                
\newcommand{\ZZ}{\mathbb{Z}}                                
\newcommand{\RR}{\mathbb{R}}                                
\newcommand{\ep}{\varepsilon}

\newcommand{\abs}[1]{\left|{#1}\right|}                     
\newcommand{\pare}[1]{\left({#1}\right)}                    
\newcommand{\set}[1]{\left\{{#1}\right\}}                   
\newcommand{\norm}[1]{\left\|{#1}\right\|}                  
\newcommand{\duality}[1]{\left<{#1}\right>}                 
\newcommand{\cl}[1]{\overline{#1}}                          
\newcommand{\restrict}{\mathord{\upharpoonright}}           
\newcommand{\compl}{\stackrel{c}{\hookrightarrow}}          

\newcommand{\lipfree}[1]{\mathcal{F}({#1})}                 
\newcommand{\lipnorm}[1]{\norm{#1}_L}                       

\renewcommand{\leq}{\leqslant}
\renewcommand{\geq}{\geqslant}

\theoremstyle{plain}
\newtheorem{theorem}{Theorem}[section]
\newtheorem{lemma}[theorem]{Lemma}
\newtheorem{corollary}[theorem]{Corollary}
\newtheorem{proposition}[theorem]{Proposition}

\newtheorem*{claim*}{Claim}

\newtheorem{innercustomthm}{Theorem}
\newenvironment{customtheorem}[1]
  {\renewcommand\theinnercustomthm{#1}\innercustomthm}
  {\endinnercustomthm}
\newtheorem{innercustomlemma}{Lemma}
\newenvironment{customlemma}[1]
  {\renewcommand\theinnercustomlemma{#1}\innercustomlemma}
  {\endinnercustomlemma}

\theoremstyle{definition}
\newtheorem*{definition*}{Definition}
\newtheorem{definition}[theorem]{Definition}
\newtheorem{example}[theorem]{Example}
\newtheorem{question}{Question}

\theoremstyle{remark}

\begin{document}

\title{Lipschitz spaces over non-porous sets}

\author[R. J. Aliaga]{Ram\'on J. Aliaga}
\address[R. J. Aliaga]{Instituto Universitario de Matem\'atica Pura y Aplicada,
Universitat Polit\`ecnica de Val\`encia,
Camino de Vera S/N,
46022 Valencia, Spain}
\email{raalva@upv.es}

\date{} 


\begin{abstract}
Let $M$ be a subset of $\mathbb{R}^n$. If $M$ is not porous, in particular if it has positive $n$-dimensional Lebesgue measure, we prove that the Lipschitz spaces $\mathrm{Lip}_0(M)$ and $\mathrm{Lip}_0(\mathbb{R}^n)$ are linearly isomorphic. The result also holds more generally if $\mathbb{R}^n$ is replaced with a Carnot group equipped with its Carnot-Carath\'eodory metric.
\end{abstract}



\maketitle
\thispagestyle{empty} 


\section{Introduction}

Let $(X,d)$ be a metric space, and fix a designated \textit{base point} $0\in X$ (we then say that the metric space is \emph{pointed}). The \emph{Lipschitz space} $\Lip_0(X)$ is the Banach space of all real-valued Lipschitz functions $f:X\to\RR$ such that $f(0)=0$, endowed with the Lipschitz norm
$$
\lipnorm{f} = \sup\set{\frac{f(x)-f(y)}{d(x,y)} \,:\, x\neq y\in X} .
$$
The space $\Lip_0(X)$ has a canonical isometric predual, the \emph{Lipschitz-free space} $\lipfree{X}$, that is generated by the evaluation functionals $\delta(x)$, $x\in X$ in $\Lip_0(X)^*$, given by $\duality{f,\delta(x)}=f(x)$.
For any $M\subset X$, McShane's extension theorem guarantees that any $f\in\Lip_0(M)$ can be extended to a function $F\in\Lip_0(X)$ without increasing its Lipschitz constant. As a consequence, $\lipfree{M}$ can be isometrically identified with a subspace of $\lipfree{X}$, namely $\cl{\lspan}\set{\delta(x):x\in M}$.
We refer to the monograph \cite{Weaver2} for further information on Lipschitz and Lipschitz-free spaces.

In this note, we focus on the case where $X$ is a Banach space with its norm metric, and ask the following question: how big must a subset $M\subset X$ be so that the Lipschitz-free, or Lipschitz, space over $M$ is isomorphic to that over $X$? Of course, the former implies the latter by taking adjoints. This question is open even in the case where $M$ and $X$ are finite-dimensional spaces: it is currently unknown whether $\lipfree{\RR^n}$ and $\lipfree{\RR^m}$, or $\Lip_0(\RR^n)$ and $\Lip_0(\RR^m)$, can be isomorphic for different values of $n,m\geq 2$. Only the one-dimensional case is solved: for $n>1$, there is no operator mapping $\Lip_0(\RR)\equiv L_\infty$ onto $\Lip_0(\RR^n)$, hence $\lipfree{\RR^n}$ does not embed into $\lipfree{\RR}\equiv L_1$ (see \cite{CDW} or \cite{NaorSchechtman}). Similarly, it is currently unknown whether $\Lip_0(X)$ can be isomorphic to some $\Lip_0(\RR^n)$ when $X$ is infinite-dimensional.

Kaufmann proved in \cite{Kaufmann} that $\lipfree{B}$ is always isomorphic to $\lipfree{X}$ for any ball $B\subset X$. If $M\subset X$ contains such a ball then we have $\lipfree{B}\subset\lipfree{M}\subset\lipfree{X}$, but we cannot immediately conclude that $\lipfree{M}$ is isomorphic to $\lipfree{X}$ unless the inclusions are complemented; if they are, we do obtain an isomorphism between $\lipfree{M}$ and $\lipfree{X}$ using Pe\l czy\'nski's decomposition method. Complementation can be guaranteed when $X$ is finite-dimensional (see Proposition \ref{pr:doubling complemented}), so we deduce that $\lipfree{M}$ is isomorphic to $\lipfree{\RR^n}$ and $\Lip_0(M)$ is isomorphic to $\Lip_0(\RR^n)$ whenever $M\subset\RR^n$ has non-empty interior.
On the other hand, Candido, C\'uth and Doucha proved in \cite{CCD} that $\Lip_0(M)$ is isomorphic to $\Lip_0(\RR^n)$ whenever $M$ is a net in $\RR^n$, such as $M=\ZZ^n$; in this case, the analogous statement for Lipschitz-free spaces is false. The takeaway seems to be that the Lipschitz space $\Lip_0(M)$ over a subset $M\subset\RR^n$ is isomorphic to $\Lip_0(\RR^n)$ provided that $M$ is ``$n$-dimensional enough''.

Our aim in this work is to make this notion of ``$n$-dimensional enough'' more precise. To that end, we make a natural further generalization and consider subsets $M\subset\RR^n$ with positive $n$-dimensional Lebesgue measure. For $n=1$, the situation is well-known: given an infinite $M\subset\RR$, the space $\lipfree{M}$ is isomorphic to either $L_1$ or $\ell_1$ depending on whether $\cl{M}$ has positive measure or not, and thus $\Lip_0(M)$ is always isomorphic to $\Lip_0(\RR)$ \cite{Godard}. Our main result sheds some light on the situation for Lipschitz spaces in dimensions greater than $1$: if $M\subset\RR^n$ has positive measure then $\Lip_0(M)$ is isomorphic to $\Lip_0(\RR^n)$ (see Corollary \ref{cr:positive measure}). In fact, we are able to show this, more generally, for any $M\subset\RR^n$ that is not porous.

\begin{customtheorem}{\ref{th:porous}}
Suppose that $M\subset\RR^n$ is not porous in $\RR^n$. Then $\Lip_0(M)$ is isomorphic to $\Lip_0(\RR^n)$.
\end{customtheorem}

Theorem \ref{th:porous} generalizes both of the aforementioned results from \cite{Kaufmann} and \cite{CCD}. Our approach towards its proof can be understood as a local version of the methods used in \cite{CCD}, where $M$ was required to possess some form of denseness behavior as well as global regularity; here, we use similar methods but we only ask that $M$ exhibits such behavior locally in infinitely many points. Our proof strategy can be summarized as follows:
\begin{itemize}
\item Because $M$ is not porous, we can find a sequence $(B_n)$ of balls in $\RR^n$ which are ``separated enough'', and such that $M$ is ``increasingly dense'' in $(B_n)$ (see Proposition \ref{pr:well separated balls}).
\item On one hand, the separation condition implies that $\Lip_0(M\cap\bigcup_n B_n)$ contains a complemented copy of the sum of the Lipschitz spaces $\Lip_0(M\cap B_n)$ (see Lemma \ref{lm:well separated free 2}).
\item On the other hand, the density condition implies that the sum of the Lipschitz spaces $\Lip_0(M\cap B_n)$ contains a complemented copy of $\Lip_0(\RR^n)$ (see Lemma \ref{lm:limit inclusion}).
\item Finally, because $\dim\RR^n<\infty$, we have enough complementation relations between the resulting Lipschitz spaces and we may apply Pe\l czy\'nski's method to conclude.
\end{itemize}

As it turns out, most of our arguments remain valid in the more general framework of Carnot groups equipped with a Carnot-Carath\'eodory metric. For the purposes of this paper, Carnot groups can be considered as noncommutative generalizations of Euclidean space; see Section \ref{sec:carnot} for further details. The properties of $\RR^n$ used in our proof coincide almost exactly with those properties characterizing Carnot groups, so we are able to extend Theorem \ref{th:porous} to that setting with very little effort (see Theorem \ref{th:porous carnot}).

The paper is structured as follows. In Section \ref{sec:porous} we consider porous sets in general and the specific separation and density conditions that will be needed in our main argument. In Section \ref{sec:main} we prove our main result about Lipschitz spaces on non-porous subsets, first for Euclidean space (Section \ref{sec:banach}) and then for Carnot groups (Section \ref{sec:carnot}). Finally, in Section \ref{sec:questions} we discuss some unanswered questions related to our research.

\subsection*{Notation}

Our notation will be standard. For a metric space $M$, the closed ball with center $x\in M$ and radius $r>0$ will be denoted by $B(x,r)$. The diameter of a subset $A\subset M$ will be denoted $\diam(A)$. Every Banach space $X$ will be tacitly regarded as a metric space with the norm metric, and its closed unit ball will be denoted by $B_X$. For $1\leq p\leq\infty$, the $\ell_p$-sum of a sequence $(X_n)$ of Banach spaces will be denoted by $\pare{\bigoplus_n X_n}_p$. Given two Banach spaces $X,Y$, we will write $X\equiv Y$ if they are linearly isometric, $X\sim Y$ if they are linearly isomorphic, and $X\compl Y$ if $Y$ contains a complemented subspace that is isomorphic to $X$. Our arguments will sometimes be based on Pe\l czy\'nski's decomposition method in the following form: if $X\compl Y\compl X$ and $X\sim\pare{\bigoplus_n X}_p$ for some $p\in [1,\infty]$, then $X\sim Y$ (see e.g. \cite[Theorem 2.2.3]{AlbiacKalton}).

The choice of different base points $0$, $0'$ in a metric space $M$ leads to isometrically isomorphic Lipschitz spaces $\Lip_0(M)$, $\Lip_{0'}(M)$ as witnessed by the operator $T:\Lip_0(M)\to\Lip_{0'}(M)$ given by $(Tf)(x)=f(x)-f(0')$. Consequently, we will sometimes omit the choice of base point or change it without further mention when we only care about the isometry or isomorphism class of a Lipschitz space. We also get isometric Lipschitz spaces $\Lip_0(M)$, $\Lip_0(N)$ whenever the metric spaces $M$ and $N$ are isometric or, more generally, related by a surjective \emph{dilation}, i.e. a mapping $\varphi:M\to N$ such that $d(\varphi(x),\varphi(y))=\lambda d(x,y)$ for some fixed $\lambda>0$. If $\varphi$ is bi-Lipschitz instead, then the Lipschitz spaces are merely isomorphic. In all of those cases, the linear isometry (isomorphism) between Lipschitz spaces is weak$^*$-to-weak$^*$ continuous and hence induces a corresponding isometry (isomorphism) between the respective Lipschitz-free spaces.

\section{Porous sets and density in balls}
\label{sec:porous}

There exist several closely related notions of porosity in the literature. The precise one that will fit our purposes reads as follows.

\begin{definition}\label{def:porous}
Let $X$ be a metric space. We say that a subset $M$ of $X$ is \emph{porous} (in $X$) if there exists $\lambda>0$ such that for every ball $B(p,r)$ in $X$, where $p\in X$ and $0<r\leq\diam(X)$, there exists $x\in X$ such that $B(x,\lambda r)\subset B(p,r)\setminus M$.
\end{definition}

\noindent Informally, ``every ball in $M$ contains a hole of comparable size''. Such sets are called ``globally very porous'' by Zaj\'i\v cek \cite{Zajicek} and simply ``porous'' by V\"ais\"al\"a \cite{Vaisala}. Let us stress that we crucially require the defining condition for porosity to hold at both small and large scales, for all radii $r>0$. For instance, nets such as $\ZZ^n$ are not porous in $\RR^n$ according to our definition, as the condition fails for large $r$. In other works, the condition is only assumed to hold for small $r$ (see e.g. \cite{Zajicek}). 

The following reformulation of porosity will be useful in our arguments.

\begin{definition}\label{def:asymptotically dense}
Let $X$ be a metric space and $(B_n)=(B(p_n,r_n))$ be a sequence of closed balls in $X$. We say that a subset $M$ of $X$ is \emph{asymptotically dense} in $(B_n)$ if there exist numbers $\ep_n>0$ such that $\ep_n\to 0$ and, for $n$ large enough, $B_n\cap M$ is $\ep_nr_n$-dense in $B_n$; that is, for each $x\in B_n$ there exists $y\in B_n\cap M$ with $d(x,y)\leq\ep_n r_n$.
\end{definition}

\noindent Clearly, if $M$ is asymptotically dense in $(B_n)$ then $M$ is asymptotically dense in any subsequence thereof as well.

If there exists a sequence of balls in $X$ as in Definition \ref{def:asymptotically dense} then $M$ is not porous in $X$, as $B_n$ witnesses the failure of Definition \ref{def:porous} for $\lambda\geq\ep_n$. The converse implication may fail in general metric spaces, but it holds e.g. when the ambient space is geodesic. Recall that a metric space $X$ is \emph{geodesic} if any pair of points in $X$ can be joined by a geodesic, i.e. an isometric copy of a closed interval in $\RR$.

\begin{lemma}
\label{lm:porous balls}
Let $X$ be a geodesic metric space. If a subset $M\subset X$ is not porous, then there exists a sequence $(B_n)$ of closed balls in $X$ such that $M$ is asymptotically dense in $(B_n)$.
\end{lemma}

We need the following simple computation.

\begin{lemma}
\label{lm:geodesic lemma}
Let $X$ be a geodesic metric space and $A,M\subset X$. Suppose that for every $p\in A$ there exists $y\in M$ with $d(p,y)\leq\delta$. Then $M\cap B$ is $2\delta$-dense in $B$ for every ball $B$ contained in $A$ with radius at least $\delta$.
\end{lemma}

\begin{proof}
Let $B=B(p,r)\subset A$ with $r\geq\delta$, and fix $x\in B$. We must show that there exists $y\in M\cap B$ such that $d(x,y)\leq 2\delta$.

Suppose first that $d(x,p)\leq\delta$. Since $p\in A$, there exists $y\in M$ such that $d(p,y)\leq\delta$, and we have $y\in B$ because $\delta\leq r$. Thus $y$ is the required point, as $d(x,y)\leq d(x,p)+d(p,y)\leq 2\delta$.

Now suppose that $d(x,p)>\delta$. Since $X$ is geodesic, there exists $q\in X$ such that $d(x,q)+d(q,p)=d(x,p)$ and $d(x,q)=\delta$. Clearly $q\in B\subset A$ as $d(p,q)<d(p,x)$, so by hypothesis there exists $y\in M$ such that $d(q,y)\leq\delta$.  We have
$$
d(y,p) \leq d(y,q)+d(q,p) \leq \delta+d(x,p)-\delta \leq r
$$
hence $y\in M\cap B$, and $d(x,y)\leq d(x,q)+d(q,y) \leq 2\delta$.
\end{proof}

\begin{proof}[Proof of Lemma \ref{lm:porous balls}]
Fix an integer $n\geq 3$. Since $M$ is not porous, the failure of Definition \ref{def:porous} for $\lambda=\frac{1}{n}$ yields $p_n\in X$ and $0<r_n\leq\diam(X)$ such that $B(x,\frac{1}{n}r_n)\not\subset B(p_n,r_n)\setminus M$ for all $x\in X$. Set $B_n=B(p_n,(1-\frac{2}{n})r_n)$. If $x\in B_n$, then $B(x,\frac{1}{n}r_n)\subset B(p_n,r_n)$ and therefore $B(x,\frac{1}{n}r_n)$ must contain an element of $M$. Thus the hypothesis of Lemma \ref{lm:geodesic lemma} is satisfied for $A=B_n$ and $\delta=\frac{1}{n}r_n$ and, since $(1-\frac{2}{n})r_n\geq\frac{1}{n}r_n$, it follows that $M\cap B_n$ is $\frac{2}{n}r_n$-dense in $B_n$. Hence $M$ is asymptotically dense in $(B_n)_{n=3}^\infty$, with constants $\ep_n=\frac{2}{n}r_n\cdot\pare{(1-\frac{2}{n})r_n}^{-1}=\frac{2}{n-2}$.
\end{proof}

For our arguments in Section \ref{sec:main}, we will require the sequence of balls $(B_n)$ to satisfy an additional separation condition. Let us formalize this notion.

\begin{definition}
We say that a collection $\mathcal{C}$ of subsets of a metric space $X$ is \emph{well-separated} with respect to $x_0\in X$ if there exists $\lambda>0$ such that
\begin{equation}
\label{eq:well separated}
d(x,y)\geq\lambda\cdot(d(x,x_0)+d(y,x_0))
\end{equation}
for any choice of $x,y$ belonging to different elements of $\mathcal{C}$. We say simply that $\mathcal{C}$ is well-separated if we do not need to specify the choice of $x_0$.
\end{definition}

\noindent Note that the intersection of any pair of elements of $\mathcal{C}$ is either empty or $\set{x_0}$. Note also that \eqref{eq:well separated} is equivalent to the simpler requirement that $d(x,y)\geq\lambda d(x,x_0)$ for some (different) $\lambda>0$.

\begin{proposition}
\label{pr:well separated balls}
Let $X$ be a complete geodesic metric space. If a subset $M\subset X$ is not porous, then there exists a sequence $(B_n)$ of pairwise disjoint, well-separated closed balls in $X$ such that $M$ is asymptotically dense in $(B_n)$.
\end{proposition}

We start by checking that, in a geodesic ambient space, we can always replace $(B_n)$ by uniformly smaller balls as follows.

\begin{lemma}
\label{lm:smaller balls}
Let $X$ be a geodesic metric space and $B_n=B(p_n,r_n)$, $n\in\NN$ be balls in $X$. Let $\lambda\in (0,1)$, and suppose that $q_n\in X$ are such that the ball $B'_n=B(q_n,\lambda r_n)$ is contained in $B_n$. If a subset $M\subset X$ is asymptotically dense in $(B_n)$, then it is also asymptotically dense in $(B'_n)$. 
\end{lemma}

\begin{proof}
Fix $\lambda\in (0,1)$, and $\ep_n\to 0$ such that $B_n\cap M$ is eventually $\ep_nr_n$-dense in $B_n$. By Lemma \ref{lm:geodesic lemma}, $B'_n\cap M$ is $2\ep_nr_n$-dense in $B'_n$ whenever $\ep_n\leq\lambda$, which holds for $n$ large enough. Since $B'_n$ has radius $\lambda r_n$ and $2\ep_nr_n=(2\lambda^{-1}\ep_n)\cdot\lambda r_n$, the condition for asymptotic density is satisfied with constants $2\lambda^{-1}\ep_n$ in place of $\ep_n$.
\end{proof}

\begin{proof}[Proof of Proposition \ref{pr:well separated balls}]
Since $M$ is not porous, by Lemma \ref{lm:porous balls} there exist balls $B_n=B(p_n,r_n)$ in $X$ and $\ep_n\in (0,1)$ such that $\ep_n\to 0$ and $B_n\cap M$ is $\ep_nr_n$-dense in $B_n$. We will use them as a starting point to construct the desired sequence of balls in $X$. Recall that the asymptotic density of $M$ is preserved if we pass to a subsequence, or reduce the radius of all balls by a constant factor (Lemma \ref{lm:smaller balls}), so we will do so frequently without further justification. Throughout the proof, we will use the notation
$$
\rad(A,x_0) = \sup\set{d(x,x_0) \,:\, x\in A}
$$
for the radius of the smallest ball centered at $x_0$ that contains the set $A\subset X$.

For our construction we consider two cases, depending on whether the sequence of radii $(r_n)$ is bounded or unbounded, and treat them separately.

\medskip

\textbf{Case 1:} $(r_n)$ is bounded. Note first that we may assume $r_n\to 0$. Indeed, consider the balls $\widetilde{B}_n=B(p_n,\tilde{r}_n)$ with $\tilde{r}_n=r_n\sqrt{\ep_n}\to 0$. We have $\tilde{r}_n\geq r_n\ep_n$, hence Lemma \ref{lm:geodesic lemma} shows that $\widetilde{B}_n\cap M$ is $\tilde{\ep}_n\tilde{r}_n$-dense in $\widetilde{B}_n$ where $\tilde{\ep}_n=2\sqrt{\ep_n}\to 0$. Thus, after passing to a subsequence to ensure that $\tilde{\ep}_n<1$ for all $n$, we may start with the balls $\widetilde{B}_n$ instead of $B_n$.

Next, we show that the $B_n$ can moreover be chosen to be pairwise disjoint. Indeed, by passing to a subsequence we assume that $r_1<\frac{1}{2}\diam(X)$ and that $r_{n+1}<\frac{1}{8}r_n$ for all $n$. By Ramsey's theorem, we may choose a further subsequence such that the balls $B'_n=B(p_n,\frac{1}{2}r_n)$ satisfy either $B'_m\cap B'_n=\varnothing$ for all $n\neq m$ or $B'_m\cap B'_n\neq\varnothing$ for all $n\neq m$. If the former holds, then the balls $B'_n$ are the ones we are seeking, so assume the latter. Then we have $B_{n+1}\subset B_n$ for all $n$: indeed, if $x\in B_{n+1}$ then fixing some $y\in B'_n\cap B'_{n+1}$ yields
$$
d(x,p_n) \leq d(x,p_{n+1})+d(p_{n+1},y)+d(y,p_n) \leq r_{n+1}+\tfrac{1}{2}r_{n+1}+\tfrac{1}{2}r_n < r_n .
$$
Now we construct a new ball $B''_n=B(q_n,\frac{1}{8}r_n)$ for each $n$. The center $q_n$ is chosen depending of two cases:
\begin{itemize}
\item If $d(p_n,p_{n+1})\geq \frac{1}{4}r_n$, let $q_n=p_n$.
\item Otherwise, let $q_n$ be any point in $X$ such that $d(p_n,q_n)=\frac{3}{4}r_n$. Note that such a point must exist: since $\diam(X)>2r_n$, there exists $z\in X\setminus B(p_n,r_n)$, and any geodesic joining $p_n$ and $z$ must contain a valid choice for $q_n$.
\end{itemize}
Note that both alternatives yield $B''_n\subset B_n\setminus B_{n+1}$ because $r_{n+1}<\frac{1}{8}r_n$. Thus the balls $B''_n$ are pairwise disjoint and, since $B''_n\subset B_n$, $M$ is asymptotically dense in $(B''_n)$ by Lemma \ref{lm:smaller balls}.

At this point, we have managed to force our sequence $B_n$ to be pairwise disjoint and satisfy $r_n\to 0$. To finish the construction, we consider three subcases.

\begin{itemize}

\item Suppose first that the sequence of points $(p_n)$ is unbounded. Fix any $x_0\in X$ and pass to a subsequence such that $d(p_n,x_0)\to\infty$ and no $B_n$ contains $x_0$. Then pass to a further subsequence such that $d(B_n,x_0)\geq 2\rad(B_k,x_0)$ for all $k<n$ (this is possible because $d(B_n,x_0)\geq d(p_n,x_0)-r_n\to\infty$). Then, given $x\in B_n,y\in B_k$ with $n>k$, we have
\begin{align*}
d(x,y) &\geq d(x,x_0)-d(y,x_0) \geq d(x,x_0)-\rad(B_k,x_0) \\
&\geq d(x,x_0)-\tfrac{1}{2}d(B_n,x_0) \geq \tfrac{1}{2}d(x,x_0) \geq \tfrac{1}{4}\pare{d(x,x_0)+d(y,x_0)}
\end{align*}
so $(B_n)$ are well-separated with constant $\lambda=\frac{1}{4}$.

\item Suppose that the sequence $(p_n)$ has an accumulation point $x_0$. Then we pass to a subsequence so that $p_n\to x_0$ and no $B_n$ contains $x_0$. Pass to a further subsequence so that $\rad(B_n,x_0)\leq\frac{1}{2}d(B_k,x_0)$ for all $k<n$ (this is possible because $\rad(B_n,x_0)\leq d(p_n,x_0)+r_n\to 0$). Then, given $x\in B_n,y\in B_k$ with $n>k$, we have
\begin{align*}
d(x,y) &\geq d(y,x_0)-d(x,x_0) \geq d(y,x_0)-\rad(B_n,x_0) \\
&\geq d(y,x_0)-\tfrac{1}{2}d(B_k,x_0) \geq \tfrac{1}{2}d(y,x_0) \geq \tfrac{1}{4}\pare{d(x,x_0)+d(y,x_0)}
\end{align*}
so $(B_n)$ are well-separated with constant $\lambda=\frac{1}{4}$.

\item If neither of the above holds then, since $X$ is complete, the set $\set{p_n:n\in\NN}$ cannot be totally bounded so we may pass to a subsequence $(p_n)$ that is bounded and uniformly discrete, i.e. there exist $R>\theta>0$ such that $\theta\leq d(p_n,p_m)\leq R$ for all $n\neq m$. Take $x_0=p_1$ and pass to a further subsequence so that $r_n\leq\frac{\theta}{4}$ for all $n$. Then, for $x\in B_n,y\in B_m$ with $n\neq m$,
$$
d(x,y) \geq d(p_n,p_m)-r_n-r_m \geq \frac{\theta}{2} 
$$
and
$$
d(x,x_0)+d(y,x_0) \leq d(x_0,p_n)+r_n+d(x_0,p_m)+r_m \leq \frac{\theta}{2}+2R
$$
so $(B_n)$ are well-separated with constant $\lambda=\frac{\theta}{\theta+4R}$.

\end{itemize}

This completes the proof of Case 1.

\medskip

\textbf{Case 2:} $(r_n)$ is unbounded. By our definition of porosity, we may assume that $r_n\leq\diam(X)$ for all $n$ and hence $X$ has infinite diameter.

We start by showing that the $(B_n)$ can be chosen to be pairwise disjoint in addition to $r_n\to\infty$. The argument is dual to that of Case 1. First, pass to a subsequence so that $r_{n+1}>8r_n$. Ramsey's theorem again yields a subsequence such that the balls $B'_n=B(p_n,\frac{1}{2}r_n)$ are either pairwise disjoint or intersect pairwise. In the former case, $(B'_n)$ is the desired sequence. Otherwise, we obtain $B_n\subset B_{n+1}$ for all $n$ as, given any $x\in B_n$ and $y\in B'_n\cap B'_{n+1}$, we have
$$
d(x,p_{n+1}) \leq d(x,p_n)+d(p_n,y)+d(y,p_{n+1}) \leq r_n+\tfrac{1}{2}r_n+\tfrac{1}{2}r_{n+1} < r_{n+1} .
$$
Next, for any $n\geq 2$, we let $B''_n=B(q_n,\frac{1}{8}r_n)$ where $q_n$ is chosen as follows:
\begin{itemize}
\item If $d(p_n,p_{n-1})\geq \frac{1}{4}r_n$, then we take $q_n=p_n$.
\item Otherwise, pick $q_n\in X$ such that $d(p_n,q_n)=\tfrac{3}{4}r_n$; note that such $q_n$ must exist on any geodesic joining $p_n$ and a sufficiently distant point.
\end{itemize}
Then we have $B''_n\subset B_n\setminus B_{n-1}$ in both alternatives, so the balls $B''_n$, $n\geq 2$ are pairwise disjoint and, since $B''_n\subset B_n$, $M$ is asymptotically dense in $(B''_n)$ by Lemma \ref{lm:smaller balls}.

So, assume that $r_n\to\infty$ and the $(B_n)$ are pairwise disjoint. To finish the construction, fix any $x_0\in X$ and pass to a subsequence such that either $d(B_n,x_0)\geq r_n$ for all $n$ or $d(B_n,x_0)\leq r_n$ for all $n$. We treat both cases separately.

\begin{itemize}

\item Suppose $d(B_n,x_0)\geq r_n$ for all $n$. Since $r_n\to\infty$, we may pass to a subsequence such that $r_n\geq 3\rad(B_k,x_0)$ whenever $n>k$. Then, for $x\in B_n,y\in B_k$ we have $d(x,x_0)\geq r_n\geq 3d(y,x_0)$ and therefore
$$
d(x,y) \geq d(x,x_0) - d(y,x_0) \geq \tfrac{1}{2}\pare{d(x,x_0) + d(y,x_0)}
$$
so the $(B_n)$ are well-separated with constant $\lambda=\frac{1}{2}$.

\item Suppose $d(B_n,x_0)\leq r_n$ for all $n$. Then we replace the balls $B_n$ with $B'_n=B(p_n,\frac{1}{2}r_n)$. For $x\in B'_n,y\in B'_m$ with $n\neq m$, we have
$$
d(x,x_0) + d(y,x_0) \leq d(B'_n,x_0)+r_n + d(B'_m,x_0)+r_m \leq 3(r_n+r_m) .
$$
On the other hand, since $X$ is geodesic, $B_n\cap B_m=\varnothing$ implies $d(p_n,p_m)>r_n+r_m$ and therefore
$$
d(x,y) \geq d(p_n,p_m) - d(x,p_n) - d(y,p_m) > \tfrac{1}{2}(r_n+r_m) .
$$
Hence the $(B'_n)$ are well-separated with constant $\lambda=\frac{1}{6}$.

\end{itemize}

This completes the construction in Case 2 and finishes the proof.
\end{proof}

\section{Proof of the main result}
\label{sec:main}

We will now see how the properties of the balls $B_n$ obtained in Proposition \ref{pr:well separated balls} allow us to infer several complementation results that will lead to the proof of our main result. Some of them will follow from technical results that are well-known to experts in the topic and easy to establish.
We collect these necessary lemmas here; constructive proofs of Lemmas \ref{lm:well separated free} and \ref{lm:limit inclusion} are provided in an Appendix.

On one hand, the balls $B_n$ are well-separated. Our main reason for considering that property is the following known result about $\ell_1$-sums of Lipschitz-free spaces. It follows from a simple computation that can be found in equivalent form in \cite[Proposition 2]{HajekNovotny}, \cite[Lemma 2.1]{AACD2}, or Proposition 5.1 in the preprint release of \cite{Kaufmann} (but not in the journal version). An early form of the statement can also be found in \cite[Proposition 5.1]{Godard}.
We refer to either of those references for the proof.

\begin{lemma}\label{lm:well separated free}
Let $M$ be a metric space and $(M_n)$ be a sequence of subsets of $M$ that are well-separated with respect to the point $x_0\in M$. Then
$$
\mathcal{F}\pare{\bigcup_{n=1}^\infty M_n\cup\set{x_0}} \sim \pare{\bigoplus_{n=1}^\infty\lipfree{M_n\cup\set{x_0}}}_1 .
$$
\end{lemma}

In order to remove the point $x_0$ from the conclusion of Lemma \ref{lm:well separated free}, we will make use of the well-known relation between the Lipschitz-free spaces over a metric space and the same metric space with one point removed.

\begin{lemma}[{\cite[Lemma 2.8]{AACD3}}]
\label{lm:removing one point}
There exists a universal constant $C<\infty$ such that, for every metric space $M$ and every $x_0\in M$, $\lipfree{M}$ is $C$-isomorphic to $\lipfree{M\setminus\set{x_0}}\oplus_1\RR$. If $M$ is infinite, then $\lipfree{M}$ is also $C$-isomorphic to $\lipfree{M\setminus\set{x_0}}$.
\end{lemma}

We shall apply Lemmas \ref{lm:well separated free} and \ref{lm:removing one point} in the following joint form.

\begin{lemma}\label{lm:well separated free 2}
Let $M$ be a metric space and $(M_n)$ be an infinite sequence of non-empty, pairwise disjoint, well-separated subsets of $M$. Then
$$
\mathcal{F}\pare{\bigcup_{n=1}^\infty M_n} \sim \pare{\bigoplus_{n=1}^\infty\lipfree{M_n}}_1 \oplus_1 \ell_1.
$$
\end{lemma}

If some $M_{n_0}$ is infinite, then $\ell_1\compl\lipfree{M_{n_0}}$ (by e.g. \cite[Theorem 1.1]{CDW}) and we get simply $\lipfree{\bigcup_n M_n}\sim\pare{\bigoplus_n\lipfree{M_n}}_1$, but our statement is also valid when all $M_n$ are finite.

\begin{proof}
Suppose the sets $(M_n)$ are well-separated with respect to $x_0\in M$. By Lemma \ref{lm:removing one point}, $\lipfree{M_n\cup\set{x_0}}$ is $C$-isomorphic to $\lipfree{M_n}\oplus_1\RR$ whenever $x_0\notin M_n$, which is the case for all $n$ except one at most. Thus Lemma \ref{lm:well separated free} implies
$$
\mathcal{F}\pare{\bigcup_{n=1}^\infty M_n\cup\set{x_0}} \sim \pare{\bigoplus_{n=1}^\infty\lipfree{M_n\cup\set{x_0}}}_1 \sim \pare{\bigoplus_{n=1}^\infty\lipfree{M_n}}_1 \oplus_1 \ell_1 .
$$
Since $\bigcup_n M_n$ is infinite, the left-hand side is isomorphic to $\lipfree{\bigcup_n M_n}$ by Lemma \ref{lm:removing one point} regardless of whether $x_0\in\bigcup_n M_n$ or not.
\end{proof}

On the other hand, the balls $B_n$ obtained in Proposition \ref{pr:well separated balls} are such that $M$ is asymptotically dense in them. If the ambient space $X$ is a Banach space, then we may rescale the $B_n$ so that they all have the same size as the unit ball $B_X$; doing so does not change their Lipschitz spaces because these are invariant under dilations. Asymptotic density then amounts to the rescalings of $B_n\cap M$ being increasingly dense subsets of $B_X$. The next lemma details the relation between the resulting Lipschitz spaces. It is a particular case of \cite[Lemma 1.3]{CCD} and similar to \cite[Theorem 3.1]{GLG}.
We refer to \cite{CCD} for the proof.

\begin{lemma}\label{lm:limit inclusion}
Let $M$ be a metric space, and let $(M_n)$ be a sequence of closed subsets of $M$ with the following property: for every $x\in M$ there exist $x_n\in M_n$ such that $x_n\to x$. Then $\Lip_0(M)$ is linearly isometric to a 1-complemented subspace of $\pare{\bigoplus_n\Lip_0(M_n)}_\infty$.
\end{lemma}

Unlike the previous lemmas, Lemma \ref{lm:limit inclusion} deals exclusively with Lipschitz spaces and its conclusion cannot possibly pass to the corresponding preduals. Indeed, if $M_n$ are increasingly dense nets in $M$, then all $\lipfree{M_n}$ have the Radon-Nikod\'ym property \cite[Proposition 4.4]{Kalton}, so $\lipfree{M}$ cannot embed into $\pare{\bigoplus_n\lipfree{M_n}}_1$ if it fails the property, e.g. if $M$ is the unit ball of a Banach space.

\subsection{The Banach space case}
\label{sec:banach}

The next proposition brings all previous results and remarks together.

\begin{proposition}\label{pr:banach inclusion}
Let $X$ be a Banach space and let $M\subset X$ be non-porous. Then there exists a subset $N\subset M$, which is also not porous in $X$, such that $\Lip_0(N)$ contains a complemented copy of $\Lip_0(X)$.
\end{proposition}

\begin{proof}
By Proposition \ref{pr:well separated balls}, there exists a sequence $(B_n)$ of pairwise disjoint, well-separated closed balls in $X$ in which $M$ is asymptotically dense. To fix notation, suppose that $B_n=B(p_n,r_n)=p_n+r_nB_X$ and $M\cap B_n$ is $\ep_n r_n$-dense in $B_n$, where $\ep_n\to 0$. Now put $M_n=r_n^{-1}(M\cap B_n - p_n)$; that is, we translate and rescale $M\cap B_n$ so that it becomes a subset of $B_X$. Note that translations in $X$ are isometries, and recall that Lipschitz spaces are isometrically invariant under isometries and rescalings of the metric space, so $\Lip_0(M_n)$ is isometric to $\Lip_0(M\cap B_n)$. Moreover $M_n$ is $\ep_n$-dense in $B_X$ and, since $\ep_n\to 0$, Lemma \ref{lm:limit inclusion} implies that $\pare{\bigoplus_n\Lip_0(M_n)}_\infty$ contains a complemented copy of $\Lip_0(B_X)$. 

Let $N=M\cap\bigcup_n B_n$. Then, taking adjoints in Lemma \ref{lm:well separated free 2} yields
\begin{align*}
\Lip_0(N) = \Lip_0\pare{\bigcup_{n=1}^\infty (M\cap B_n)} &\sim \pare{\bigoplus_{n=1}^\infty\Lip_0(M\cap B_n)}_\infty \oplus_\infty \ell_\infty \\
&\equiv \pare{\bigoplus_{n=1}^\infty\Lip_0(M_n)}_\infty \oplus_\infty \ell_\infty ,
\end{align*}
hence $\Lip_0(N)$ contains a complemented copy of $\Lip_0(B_X)$. Finally, note that $\Lip_0(B_X)$ is isomorphic to $\Lip_0(X)$ by \cite[Corollary 3.3]{Kaufmann}, and that $N$ is also a non-porous subset of $X$ as it is asymptotically dense in $(B_n)$.
\end{proof}

Let us see how this can be improved when $X$ is finite-dimensional. Recall that a metric space $X$ is \emph{doubling} if there exists a constant $N\in\NN$ such that, for every $r>0$, every closed ball in $X$ with radius $r$ can be covered with at most $N$ closed balls with radius $r/2$. Euclidean spaces are the prototypical examples, and it is easy to check that subspaces of doubling spaces are again doubling. In such a setting, containment between Lipschitz-free spaces is always complemented. This follows from the work of Lee and Naor \cite{LeeNaor} and was first observed by Lancien and Perneck\'a \cite{LancienPernecka}.

\begin{proposition}[{\cite[Proposition 1.8]{CCD}}]
\label{pr:doubling complemented}
Let $X$ be a doubling metric space. Then, for any choice of subsets $N\subset M\subset X$, we have $\lipfree{N}\compl\lipfree{M}$ and $\Lip_0(N)\compl\Lip_0(M)$.
\end{proposition}

This property holds, more generally, for spaces with finite Nagata dimension \cite[Lemma 3.2]{FG2} but we will not need such generality.

Our main theorem can now be obtained by combining Propositions \ref{pr:banach inclusion} and \ref{pr:doubling complemented}.

\begin{theorem}
\label{th:porous}
Suppose that $M\subset\RR^n$ is not porous in $\RR^n$. Then $\Lip_0(M)$ is isomorphic to $\Lip_0(\RR^n)$.
\end{theorem}

\begin{proof}
Let $M\subset\RR^n$ be non-porous. On one hand, Proposition \ref{pr:banach inclusion} yields a subset $N\subset M$ such that $\Lip_0(\RR^n)\compl\Lip_0(N)$. On the other, we have $\Lip_0(N)\compl\Lip_0(M)\compl\Lip_0(\RR^n)$ by Proposition \ref{pr:doubling complemented}. Finally, $\Lip_0(\RR^n)$ is isomorphic to its own $\ell_\infty$-sum by \cite[Theorem 3.1]{Kaufmann}. Thus, all requirements for application of Pe\l czy\'nski's decomposition method are satisfied, and we conclude $\Lip_0(M)\sim\Lip_0(\RR^n)$.
\end{proof}

It is well known that porous subsets of $\RR^n$ are Lebesgue null. Indeed, if $M\subset\RR^n$ is porous then there exists $\lambda>0$ such that every ball $B$ in $\RR^n$ with radius $r$ contains a ball $B'$ with radius $\lambda r$ that does not intersect $M$, therefore $\mathcal{L}(M\cap B)\leq \mathcal{L}(B\setminus B')=(1-\lambda^n)\mathcal{L}(B)$ where $\mathcal{L}$ stands for $n$-dimensional Lebesgue measure. Thus $M$ contains no point of density. So we obtain the following particular case.

\begin{corollary}
\label{cr:positive measure}
Suppose that $M\subset\RR^n$ has positive $n$-dimensional Lebesgue measure. Then $\Lip_0(M)$ is isomorphic to $\Lip_0(\RR^n)$.
\end{corollary}

\subsection{The Carnot group case}
\label{sec:carnot}

Let us review the main properties of $\RR^n$ that have been used in the proof of Theorem \ref{th:porous}. First, we used the fact that it is geodesic and complete so we could apply Proposition \ref{pr:well separated balls} to extract the sequence of balls $(B_n)$. Then, in Proposition \ref{pr:banach inclusion}, we used its invariance with respect to translations and rescalings in order to identify all of these balls with each other. From a purely metric point of view, this means that $\RR^n$ is metrically \emph{homogeneous} (for any $x,y\in\RR^n$ there is a bijective isometry on $\RR^n$ taking $x$ to $y$) and \emph{self-similar} (for any $\lambda>0$ there is a bijective dilation on $\RR^n$ with factor $\lambda$). Lastly, we used the fact that $\RR^n$ is doubling in order to obtain complementability via Proposition \ref{pr:doubling complemented}.

A more general class of metric spaces satisfying all of the above are Carnot groups. A \emph{Carnot group} $G$ is a simply connected Lie group whose associated Lie algebra $\mathfrak{g}$ admits a stratification, i.e. a finite direct sum decomposition $\mathfrak{g}=V_1\oplus\ldots\oplus V_n$ such that $[V_i,V_1]=V_{i+1}$ and $V_{n+1}=\set{0}$. Carnot groups can be canonically endowed with left-invariant (Finsler-)Carnot-Carath\'eodory metrics that are unique up to bi-Lipschitz equivalence (left invariance meaning that $d(z\cdot x,z\cdot y)=d(x,y)$ for all $x,y,z\in G$), and admit for any $\lambda>0$ a dilation $\delta_\lambda$ fixing the identity element such that $d(\delta_\lambda(x),\delta_\lambda(y))=\lambda d(x,y)$ for all $x,y\in G$. Finite-dimensional Banach spaces are precisely the abelian Carnot groups, but there are also non-abelian examples, the most prominent of which is probably the Heisenberg group. We refer to \cite{LeDonnePrimer} for a more detailed introduction to these spaces.

This description suggests that Theorem \ref{th:porous} should also hold when $\RR^n$ is replaced with a Carnot group. In fact, it was observed by Le Donne in \cite{LeDonne} that the required properties (being geodesic, doubling, homogeneous and self-similar) characterize Carnot groups among metric spaces, so we cannot hope to extend our argument beyond that case without substantial changes.

There is, in fact, one additional property of $\RR^n$ (or any Banach space $X$) that has been used in our main proof, due to Kaufmann: the fact that $\lipfree{X}$ is isomorphic to $\lipfree{B_X}$ and also to its own $\ell_1$-sum \cite[Theorem 3.1 and Corollary 3.3]{Kaufmann}. In order to extend Theorem \ref{th:porous}, we need the analog of Kaufmann's result for Carnot groups. This can be found in \cite[Section 5]{AACD3}; the isomorphism $\lipfree{G}\sim\pare{\bigoplus_{n\in\NN}\lipfree{G}}_1$ can also be obtained from \cite[Theorem 1.13]{CCD}. A self-contained proof is provided in the Appendix.

\begin{lemma}[{\cite[Corollary 5.5 and Theorem 5.8]{AACD3}}]\label{lm:carnot ball}
Let $G$ be a Carnot group equipped with its Carnot-Carath\'eodory metric, and let $B$ be any closed ball in $G$ with positive radius. Then the Banach spaces $\lipfree{B}$, $\lipfree{G}$ and $\pare{\bigoplus_{n\in\NN}\lipfree{G}}_1$ are isomorphic.
\end{lemma}

\begin{theorem}
\label{th:porous carnot}
Let $G$ be a Carnot group equipped with its Carnot-Carath\'eodory metric. Suppose that $M\subset G$ is not porous. Then $\Lip_0(M)$ is isomorphic to $\Lip_0(G)$.
\end{theorem}

\begin{proof}
The argument is essentially the same as in Proposition \ref{pr:banach inclusion} and Theorem \ref{th:porous}. By Proposition \ref{pr:well separated balls}, there is a sequence of pairwise disjoint, well-separated balls $B_n=B(p_n,r_n)$ in $G$ such that $M\cap B_n$ is $\ep_n r_n$-dense in $B_n$ for values $\ep_n\to 0$. Put $B=B(0,1)$ where $0$ is the identity element of $G$, and let
$$
M_n = \delta_{1/r_n}(p_n^{-1}\cdot (M\cap B_n)) = \set{\delta_{1/r_n}(p_n^{-1}\cdot x) \,:\, x\in M\cap B_n} ,
$$
where $\cdot$ is the group operation of $G$ and $p_n^{-1}$ is the inverse of $p_n$ under said operation. This transformation is a dilation so we have $\Lip_0(M_n)\equiv\Lip_0(M\cap B_n)$, and Lemma \ref{lm:well separated free 2} implies that $\pare{\bigoplus_n\Lip_0(M_n)}_\infty\compl\Lip_0(N)$ where $N=M\cap\bigcup_n B_n$. Moreover, $M_n$ is a subset of $B$ that is $\varepsilon_n$-dense in $B$, thus by Lemma \ref{lm:limit inclusion} we have $\Lip_0(B)\compl\pare{\bigoplus_n\Lip_0(M_n)}_\infty$. We also have $\Lip_0(N)\compl\Lip_0(M)\compl\Lip_0(G)$ by Proposition \ref{pr:doubling complemented}. To recap,
$$
\Lip_0(B) \compl \pare{\bigoplus_{n=1}^\infty\Lip_0(M_n)}_\infty \compl \Lip_0(N) \compl \Lip_0(M) \compl \Lip_0(G) .
$$
Finally, we also have $\Lip_0(B)\sim\Lip_0(G)\sim\pare{\bigoplus_n\Lip_0(G)}_\infty$ by Lemma \ref{lm:carnot ball}, and so Pe\l czy\'nski's decomposition method yields $\Lip_0(M)\sim\Lip_0(G)$.
\end{proof}

\section{Questions and discussion}
\label{sec:questions}

We are afraid that our result opens up more questions than it closes. The main open question is the one that motivated our research in the first place:

\begin{question}\label{q:positive measure}
Suppose that $M\subset\RR^n$ has positive Lebesgue measure. Is $\lipfree{M}$ is isomorphic to $\lipfree{\RR^n}$?
\end{question}

\noindent The answer is clearly false for non-porous $M$, as e.g. $\lipfree{\ZZ^n}$ cannot be isomorphic to $\lipfree{\RR^n}$ (the former has the Radon-Nikod\'ym property while the latter contains $\lipfree{\RR}\equiv L_1$) even though their duals are isomorphic. But one could possibly give a positive answer to Question \ref{q:positive measure} using a different method. Most steps of our argument would remain valid for Lipschitz-free spaces in place of Lipschitz spaces, with the only exception of Lemma \ref{lm:limit inclusion}.

A natural question for Lipschitz spaces is whether Theorem \ref{th:porous} holds in arbitrary Banach spaces.

\begin{question}
Let $X$ be a Banach space. If $M\subset X$ is not porous, is $\Lip_0(M)$ isomorphic to $\Lip_0(X)$?
\end{question}

\noindent Our proof of Theorem \ref{th:porous} requires the ambient space $X$ to be doubling, hence finite-dimensional, in order to deduce $\Lip_0(M)\compl\Lip_0(X)$ via Proposition \ref{pr:doubling complemented}. The usual argument for the proof of Proposition \ref{pr:doubling complemented} involves the existence of bounded linear extension operators $\Lip_0(M)\to\Lip_0(X)$, but such operators do not exist in general for infinite-dimensional $X$, for instance for $X=\ell_1$ (see e.g. the discussion in \cite{APQ} after Proposition 2.11).

If we stick to finite dimension, even to $\RR^2$, only two different (infinite-dimensional) Lipschitz spaces are known: $\Lip_0(\RR)\equiv\ell_\infty$, and $\Lip_0(\RR^2)$. It is not known whether these are the only possibilities, or whether there is a third (let alone infinitely many) isomorphism class of Lipschitz spaces over two-dimensional sets.

\begin{question}\label{q:subsets R2}
Let $M\subset\RR^2$ be infinite. Must $\Lip_0(M)$ be isomorphic to either $\Lip_0(\RR)$ or $\Lip_0(\RR^2)$? If not, how many isomorphism classes are there?
\end{question}

\noindent A similar question could of course be asked for $\RR^n$ in place of $\RR^2$, but the more urgent question whether $\Lip_0(\RR^n)$ and $\Lip_0(\RR^k)$ are non-isomorphic for $n>k\geq 2$ is currently still open.

Regardless of whether Question \ref{q:subsets R2} has a positive or negative answer, one would naturally want to find a metric characterization for $M$ determining when its Lipschitz space belongs to one isomorphism class or another. Theorem \ref{th:porous} probably provides the most general condition to date, but we do not know whether it is a characterization.

\begin{question}
Is there a porous subset $M$ of $\RR^2$ such that $\Lip_0(M)$ is isomorphic to $\Lip_0(\RR^2)$?
\end{question}

Classification results should be easier to obtain for Lipschitz spaces than for Lipschitz-free spaces, as the former appear to be less diverse. Recent efforts along this path can be found e.g. in \cite{CCD,CandidoGuzman,CandidoKaufmann}. But problems such as Question \ref{q:subsets R2} and the corresponding metric characterizations can of course be asked for Lipschitz-free spaces as well. For instance, to the best of our knowledge there are exactly four known infinite-dimensional Lipschitz-free spaces over subsets of $\RR^2$, up to isomorphism: $\lipfree{\ZZ}\equiv\ell_1$, $\lipfree{\RR}\equiv L_1$, $\lipfree{\ZZ^2}$, and $\lipfree{\RR^2}$. The Radon-Nikod\'ym property and the results of Naor and Schechtman \cite{NaorSchechtman} are enough to tell all of these apart from each other. But we do not know whether there is a fifth class. Even for some specific, very simple subsets, we cannot currently tell what their Lipschitz-free space is. For example, the same criteria as above imply that $\lipfree{\RR\times\ZZ}$ cannot be isomorphic to the first three spaces in the previous list, but we do not know about the last one.

\begin{question}
Is $\lipfree{\RR\times\ZZ}$ isomorphic to $\lipfree{\RR^2}$?
\end{question}

\noindent Our only knowledge about $\lipfree{\RR\times\ZZ}$ is that it is isomorphic to $\lipfree{\RR}\oplus\lipfree{\ZZ^2}$ \cite[Proposition 3.14]{AliagaMedina}.

We wish to highlight one last example whose Lipschitz-free space is unknown, that has appeared frequently in conversations with colleagues but never in print as far as we know.

\begin{example}
Let $(a_n)$ be an increasing sequence of non-negative numbers and consider the set $M=\set{(a_m,a_n)\in\RR^2 : m,n\in\NN}$. If $a_n$ grows linearly, then $M$ is a net in $\RR^+\times\RR^+$ and the techniques from \cite{HajekNovotny} show that $\lipfree{M}\sim\lipfree{\ZZ^2}$ (indeed, we have $\lipfree{M}\sim\lipfree{\NN^2}$ by \cite[Proposition 5]{HajekNovotny} and $\lipfree{\NN^2}\sim\lipfree{\ZZ^2}$ as a byproduct of the proof of \cite[Theorem 7]{HajekNovotny}, or alternatively by \cite[Corollary 5.12]{AACD3}). On the other hand, if $a_n$ grows exponentially then we claim that $\lipfree{M}\sim\ell_1$.

We only sketch the argument for the particular case $a_n=2^n$. Endow $M$ with the $\ell_1$ distance inherited from $\RR^2$ (which does not change the isomorphism class of $\lipfree{M}$), and put
$$
A = \set{(a_m,a_1) \,:\, m\in\NN} \cup \set{(a_1,a_n) \,:\, n\in\NN} .
$$
Define a retraction $r:M\to A$ by
$$
r(a_m,a_n) = \begin{cases} (a_m,a_1) &\text{, if }a_m\geq a_n \\ (a_1,a_n) &\text{, if }a_m<a_n \end{cases}
$$
and note that it is a nearest-point map, i.e. $d(x,r(x))=d(x,A)$ for all $x\in M$. It is easily checked that $r$ is Lipschitz. Indeed, fix $m,n,m',n'\in\NN$. If $a_m\geq a_n$ and $a_{m'}\geq a_{n'}$ then
$$
d(r(a_m,a_n),r(a_{m'},a_{n'})) = d((a_m,a_1),(a_{m'},a_1)) = \abs{a_m-a_{m'}} \leq d((a_m,a_n),(a_{m'},a_{n'})) .
$$
Suppose $a_m\geq a_n$ and $a_{m'}<a_{n'}$, then
$$
d(r(a_m,a_n),r(a_{m'},a_{n'})) = d((a_m,a_1),(a_1,a_{n'})) \leq a_m+a_{n'} \leq 2\max\set{a_m,a_{n'}} .
$$
If $a_m\geq a_{n'}$ then $a_m>a_{m'}$, and this value is $2a_m\leq 4\abs{a_m-a_{m'}}$ due to our specific choice $a_m=2^m$. Otherwise $a_m<a_{n'}$ and $a_{n'}>a_n$, and the bound is $2a_{n'}\leq 4\abs{a_{n'}-a_n}$. Thus, in any case
$$
d(r(a_m,a_n),r(a_{m'},a_{n'})) \leq 4\pare{\abs{a_m-a_{m'}}+\abs{a_n-a_{n'}}} = 4\,d((a_m,a_n),(a_{m'},a_{n'})) .
$$
The remaining two cases are handled similarly. We conclude that $r$ is a $4$-Lipschitz retraction onto $A$. By \cite[Proposition 1]{HajekNovotny}, we deduce $\lipfree{M}\sim\lipfree{A}\oplus\lipfree{M/A}$ where $M/A$ is the quotient metric space $(M\setminus A)\cup\set{A}$ endowed with the quotient metric given by $d_{M/A}(x,A)=d(x,A)$ and
$$
d_{M/A}(x,y) = \min\set{d(x,y),d(x,A)+d(y,A)}
$$
for $x,y\in M\setminus A$ (see e.g. \cite[Proposition 1.26]{Weaver2}). Now, the nearest-point map $r$ being $4$-Lipschitz implies that the collection of all singletons of $M/A$ is well-separated with constant $\lambda=\frac{1}{4}$ with respect to the base point $A$, and therefore $\lipfree{M/A}\sim\ell_1$ by Lemma \ref{lm:well separated free}. On the other hand, $A$ is isometric to a countable subset of $\RR$ and therefore $\lipfree{A}\equiv\ell_1$ by e.g. \cite[Corollary 3.4]{Godard}. We conclude that $\lipfree{M}\sim\ell_1$. The argument for a different exponentially growing sequence is the same, possibly with a different Lipschitz constant for the retraction $r$.

This begs the question: what happens in an intermediate case where $a_n$ grows sub-exponentially but faster than linearly, for instance when it has polynomial growth? Our main result allows us to conclude that $\lipfree{M}$ cannot be isomorphic to $\ell_1$ in that case, as $\Lip_0(M)$ is isomorphic to $\Lip_0(\RR^2)$. Let us show this for the simplest example $a_n=n^2$. Endow $\RR^2$ with the $\ell_\infty$ distance and consider the ball $B=B((n^2,n^2),n^2)$ for $n\in\NN$. Then $B\cap M$ consists exactly of the points $(k^2,l^2)$ with $1\leq k,l\leq m$, where $m$ is the largest integer with $m^2\leq 2n^2$. Note that the largest possible distance between any adjacent pair of points of $M$ with coordinates bounded by $(m+1)^2$ is $(m+1)^2-m^2=2m+1\leq 2\sqrt{2}n+1\leq 4n$. Thus, any ball in $\RR^2$ whose center belongs to $B$ and whose radius is at least $4n$ must intersect $B\cap M$. So $B$ witnesses that $M$ fails the condition from Definition \ref{def:porous} for all $\lambda\geq\frac{4n}{n^2}=\frac{4}{n}$. Since this is true for any $n$, $M$ cannot be porous and Theorem \ref{th:porous} yields $\Lip_0(M)\sim\Lip_0(\RR^2)$. This argument can be easily generalized to any sequence of the form $a_n=p(n)$ where $p$ is a polynomial; in that case, $a_n$ grows as $n^d$ but the gaps $a_n-a_{n-1}$ grow as $n^{d-1}$, where $d$ is the degree of $p$.
\end{example}

\begin{question}
Let $M=\set{(p(m),p(n))\,:\,m,n\in\NN}\subset\RR^2$ where $p$ is a polynomial (for instance, $M=\set{(m^2,n^2)\,:\,m,n\in\NN}\subset\ZZ^2$). Is $\lipfree{M}$ isomorphic to $\lipfree{\ZZ^2}$?
\end{question}

\section*{Appendix}

We include here, for the benefit of the reader, constructive proofs of three of the technical lemmas included in Section \ref{sec:main}.

\begin{customlemma}{\ref{lm:well separated free}}
Let $M$ be a metric space and $(M_n)$ be a sequence of subsets of $M$ that are well-separated with respect to the point $x_0\in M$. Then
$$
\mathcal{F}\pare{\bigcup_{n=1}^\infty M_n\cup\set{x_0}} \sim \pare{\bigoplus_{n=1}^\infty\lipfree{M_n\cup\set{x_0}}}_1 .
$$
\end{customlemma}

\begin{proof}
We may take $x_0$ as the base point of $M$, as that does not change the isometry class of the Lipschitz-free spaces. Suppose that the sets $(M_n)$ satisfy \eqref{eq:well separated} for a certain $\lambda\leq 1$, and write $M'_n=M_n\cup\set{x_0}$ and $N=\bigcup_n M_n\cup\set{x_0}$.

Define an operator $R:\Lip_0(N)\to\pare{\bigoplus_n\Lip_0(M'_n)}_\infty$ by $Rf=(f\restrict_{M'_n})_{n=1}^\infty$; it is clear that $R$ is linear and injective and $\norm{R}=1$. Next, define a mapping $T:\pare{\bigoplus_n\Lip_0(M'_n)}_\infty\to\Lip_0(N)$ by setting $T((f_n))\restrict_{M'_n}=f_n$; this is well-defined as the sets $M'_n$ only intersect at $x_0$, where every $f_n$ takes the value $0$. It is clear that $T$ is linear and injective. Suppose that $\lipnorm{f_n}\leq 1$ for all $n$ and let $x,y\in N$. If $x,y$ belong to the same set $M'_n$ then
$$
\abs{T((f_n))(x)-T((f_n))(y)} = \abs{f_n(x)-f_n(y)} \leq d(x,y) ,
$$
and if $x\in M_k$ and $y\in M_l$, $k\neq l$, we have
$$
\abs{T((f_n))(x)-T((f_n))(y)} = \abs{f_k(x)-f_l(y)} \leq \abs{f_k(x)} + \abs{f_l(y)} \leq d(x,x_0) + d(y,x_0) \leq \frac{1}{\lambda}d(x,y) .
$$
Therefore $\lipnorm{T((f_n))}\leq\lambda^{-1}$ and we conclude that $T$ is a bounded operator with norm at most $\lambda^{-1}$. It is also clear that $R$ and $T$ are inverses of each other, thus $\Lip_0(N)$ and $\pare{\bigoplus_n\Lip_0(M'_n)}_\infty$ are $\lambda^{-1}$-isomorphic. In order to prove the same for their preduals $\lipfree{N}$ and $\pare{\bigoplus_n\lipfree{M'_n}}_1$, we only need to check that $R$ and $T$ are weak$^*$-to-weak$^*$ continuous. By the Banach-Dieudonn\'e theorem, it suffices to check that they are pointwise-to-pointwise continuous, but this is obvious from the definition. The isomorphism is thus established.
\end{proof}

\begin{customlemma}{\ref{lm:limit inclusion}}
Let $M$ be a metric space, and let $(M_n)$ be a sequence of closed subsets of $M$ with the following property: for every $x\in M$ there exist $x_n\in M_n$ such that $x_n\to x$. Then $\Lip_0(M)$ is linearly isometric to a 1-complemented subspace of $\pare{\bigoplus_n\Lip_0(M_n)}_\infty$.
\end{customlemma}

\begin{proof}
Recall that the spaces $\Lip_0(M)$ are isometric for any choice of base point $0\in M$. We may therefore choose base points $0\in M$, $0_n\in M_n$ such that $0_n\to 0$. Denote $Z=\pare{\bigoplus_n\Lip_{0_n}(M_n)}_\infty$.

First, we define a linear mapping $R:\Lip_0(M)\to Z$ by $R(f)_n=f\restrict_{M_n}-f(0_n)$. Suppose that $f\in\Lip_0(M)$ with $\lipnorm{f}=1$, then we can find a sequence $(x_k,y_k)$ of pairs of different points in $M$ such that $f(x_k)-f(y_k)\geq (1-\frac 1k)d(x_k,y_k)$. For each $k$, there exist by assumption an index $n\in\NN$ and $u,v\in M_n$ such that $d(x_k,u),d(y_k,v)\leq \frac 1k d(x_k,y_k)$. Thus
$$
f(u)-f(v) \geq f(x_k)-f(y_k)-\tfrac 2k d(x_k,y_k) \geq (1-\tfrac 3k)d(x_k,y_k) \geq (1-\tfrac 3k)(1+\tfrac 2k)^{-1}d(u,v)
$$
and $\lipnorm{f\restrict_{M_n}}\geq (1-\frac 3k)(1+\frac 2k)^{-1}$. Letting $k\to\infty$, we conclude $\norm{Rf}\geq 1$. It is also clear that $\norm{Rf}\leq 1$, so $R$ is a linear isometry.

Next, fix a free ultrafilter $\mathcal{U}$ on $\NN$, and define a mapping $S:Z\to\Lip_0(M)$ as follows. Given $\mathbf{f}=(f_n)\in Z$ with $\norm{\mathbf{f}}=1$, use McShane's theorem to extend each $f_n\in\Lip_{0_n}(M_n)$ to a function $F_n\in\Lip_{0_n}(M)$ with $F_n\restrict_{M_n}=f_n$ and $\lipnorm{F_n}=\lipnorm{f_n}$, and then set
$$
(Sf)(x) = \lim_{\mathcal{U},n} F_n(x)
$$
for $x\in M$. Let us check that $S$ is well defined. The limit clearly exists for each $x\in M$, as $\abs{F_n(x)}\leq d(x,0_n)$ for all $n$. To see that it does not depend on the choice of $F_n$, fix points $x_n\in M_n$ such that $x_n\to x$ and note that $\abs{F_n(x)-F_n(x_n)} \leq d(x,x_n) \to 0$, so the limit is uniquely determined by the values $F_n(x_n)=f_n(x_n)$. This also shows that $(Sf)(0) = \lim_n f_n(0_n) = 0$. Given $x,y\in M$, we have
$$
\abs{(Sf)(x)-(Sf)(y)} = \lim_{\mathcal{U},n}\abs{F_n(x)-F_n(y)} \leq d(x,y)
$$
and therefore $Sf\in\Lip_0(M)$ with $\lipnorm{Sf}\leq 1$. It is also clear that $S$ is linear, so $S$ is a well-defined operator with $\norm{S}\leq 1$.

Given $f\in\Lip_0(M)$, we clearly have $S(Rf)=f$ as we can choose $f-f(0_n)$ as the extension of each coordinate $f\restrict_{M_n}-f(0_n)$ of $R(f)$. Thus $SR$ is the identity, and it follows that $RS$ is a projection of $Z$ onto the isometric copy $R(\Lip_0(M))$ of $\Lip_0(M)$, with $\norm{RS}\leq\norm{R}\norm{S}\leq 1$. Thus $Z$ contains a $1$-complemented, isometric copy of $\Lip_0(M)$.
\end{proof}

\begin{customlemma}{\ref{lm:carnot ball}}
Let $G$ be a Carnot group equipped with its Carnot-Carath\'eodory metric, and let $B$ be any closed ball in $G$ with positive radius. Then the Banach spaces $\lipfree{B}$, $\lipfree{G}$ and $\pare{\bigoplus_{n\in\NN}\lipfree{G}}_1$ are isomorphic.
\end{customlemma}

\begin{proof}
For the proof, we need to recall a general decomposition result originally due to Kalton \cite{Kalton}, although we will use it in its slightly simpler formulation given in \cite{AP_normality}. Given any pointed metric space $M$ with base point $0$, for $n\in\ZZ$ let $\Lambda_n\in\Lip_0(M)$ be defined by
$$
\Lambda_n(x) = \begin{cases}
2^{-(n-1)}d(x,0)-1 &\text{, if } 2^{n-1}\leq d(x,0)\leq 2^n \\
2-2^{-n}d(x,0) &\text{, if } 2^n\leq d(x,0)\leq 2^{n+1} \\
0 &\text{, otherwise}
\end{cases}
$$
and consider the mapping $W_n:\lipfree{M}\to\lipfree{M}$ given by $\duality{f,W_n(\mu)}=\duality{\Lambda_n\cdot f,\mu}$ for $\mu\in\lipfree{M}$, $f\in\Lip_0(M)$. Then $W_n$ is a bounded linear operator, its range is contained in $\lipfree{R_n\cup\set{0}}$ where
$$
R_n = \set{x\in M \,:\, 2^{n-1}\leq d(x,0)\leq 2^{n+1}} ,
$$
and every $\mu\in\lipfree{M}$ satisfies $\mu=\sum_{n\in\ZZ} W_n(\mu)$, where the series converges absolutely with $\sum_{n\in\ZZ}\norm{W_n(\mu)}\leq 45\norm{\mu}$ (see \cite[Lemma 3]{AP_normality}). From here, it follows easily that
$$
\lipfree{M}\compl\pare{\bigoplus_{n\in\ZZ}\lipfree{R_n\cup\set{0}}}_1 .
$$
Indeed, let $T:\lipfree{M}\to (\bigoplus_n\lipfree{R_n\cup\set{0}})_1$ and $S:(\bigoplus_n\lipfree{R_n\cup\set{0}})_1\to\lipfree{M}$ be defined by $T\mu=(W_n(\mu))_{n\in\ZZ}$ and $S((\mu_n)_n)=\sum_n\mu_n$. Then $S,T$ are bounded operators with $\norm{S}\leq 1$ and $\norm{T}\leq 45$ and $ST$ is the identity on $\lipfree{M}$, so $TS$ is the desired projection onto a subspace isomorphic to $\lipfree{M}$.

Now suppose that $M=G$ is a Carnot group and let $0$ be the identity element. Recall that $G$ is homogeneous and self-similar, therefore the space $\lipfree{B}$ is uniquely determined up to linear isometry for any closed ball $B\subset G$, so we may assume $B=B(0,1)$. The spaces $\lipfree{R_n\cup\set{0}}$ are all isometric to each other for the same reason, as $R_n\cup\set{0}=\delta_{2^{n-m}}(R_m\cup\set{0})$ for any $n,m\in\ZZ$. Thus, by Proposition \ref{pr:doubling complemented} we have
$$
\lipfree{B} \compl \lipfree{G} \compl \pare{\bigoplus_{n\in\ZZ}\lipfree{R_n\cup\set{0}}}_1 \equiv \pare{\bigoplus_n\lipfree{R}}_1
$$
where $R=R_0\cup\set{0}$. On the other hand, let $A=\bigcup_{n=1}^\infty R_{-3n}\cup\set{0}$. Given any $x\in R_{-3n}$, $y\in R_{-3m}$ with $1\leq n<m$ we have
$$
d(x,y) \geq d(x,0) - d(y,0) \geq 2^{-3n-1}-2^{-3m+1} \geq 2^{-3n-1}-2^{-3(n+1)+1} = 2^{-3n-2}
$$
and
$$
d(x,0)+d(y,0) \leq 2^{-3n+1}+2^{-3m+1} \leq 2\cdot 2^{-3n+1} = 2^{-3n+2} ,
$$
therefore the sets $(R_{-3n}\cup\set{0})$ are well-separated with respect to $0$ with constant $\lambda=\frac{1}{16}$. It then follows from Lemma \ref{lm:well separated free} that
$$
\pare{\bigoplus_n\lipfree{R}}_1 \equiv \pare{\bigoplus_{n\in\NN}\lipfree{R_{-3n}\cup\set{0}}}_1 \sim \lipfree{A} \compl \lipfree{B}
$$
where the last relation follows from $A\subset B$ and Proposition \ref{pr:doubling complemented}. Since $\pare{\bigoplus_n\lipfree{R}}_1$ is clearly isomorphic to its countable $\ell_1$-sum, Pe\l czy\'nski's decomposition method shows that $\lipfree{B}$ and $\lipfree{G}$ are both isomorphic to $\pare{\bigoplus_n\lipfree{R}}_1$, and therefore also to their own countable $\ell_1$-sums.
\end{proof}

\section*{Acknowledgments}

This research was conceived while the author was visiting the Faculty of Information Technology at the Czech Technical University in Prague in 2024. The author would like to thank the members of the Lipschitz-free space seminar in Prague, where the results in this paper were first pitched: Jan B\'ima, Marek C\'uth, Michal Doucha, Eva Perneck\'a, Tom\'a\v s Raunig, Jan Sp\v ev\'ak, and Henrik Wirzenius. He would also like to thank Christian Bargetz and Chris Gartland for kindly answering his questions on the topics of porous sets and Carnot groups, respectively.

The author was partially supported by Grant PID2021-122126NB-C33 funded by MICIU/AEI/ 10.13039/501100011033 and by ERDF/EU.


\end{document}